\documentclass[11pt]{article}

\usepackage[pagewise]{lineno}
\usepackage{amssymb}
\usepackage{latexsym,amsfonts,amssymb,amsmath,amsthm}
\usepackage{graphicx}
\usepackage{cite}
\usepackage{amsthm}
\usepackage{amsfonts}
\usepackage{xr}
\usepackage{graphicx}
\usepackage{amsmath}
\usepackage{color}
\usepackage{amssymb}
\usepackage{mdwlist}

\newtheorem{theorem}{Theorem}[section]
\newtheorem{definition}{Definition}[section]

\newtheorem{lemma}[theorem]{Lemma}
\newtheorem{rem}[theorem]{Remark}

\numberwithin{equation}{section}

\makeatletter 
\@addtoreset{equation}{section}
\makeatother  

\newcommand\norm[1]{\lVert#1\rVert}
\newcommand\abs[1]{\lvert#1\rvert}

\begin{document}

\title{Decomposition of quasi-all invariant bundles for the cocycle with a set of strongly invariant cones}
\author{Lirui Feng\thanks{School of Mathematical Sciences, University of Science and Technology of China, Hefei, Anhui, 230026, People's Republic of China (ruilif@ustc.edu.cn,\,ruilif@163.com), Supported by NSF of China No.12331006 and the Chinese Academy of Sciences through grant: XDB0900100 (Sub-project: XDB0900101).}}
\date{}
\maketitle
\begin{abstract}
In this paper, we prove that a linear cocycle with a set of strongly invariant cones admits a decomposition of quasi-all invariant bundles. We further obtain the certain location-relations between the fibres of positive invariant bundles and these cones, as well as we get the order-relations among characteristic exponents of the linear cocycle on the positive invariant bundles. The key of proofs is to analyze relations among positive invariant bundles with respect to a linear cocycle that admits both of a $k$-exponential separation and a $k^{\prime}$-exponential separation. These results are applied to a class of semilinear equations on a Hilbert space.
\end{abstract}

\section{Introduction} The linear cocylce $(F,\mathcal{T})$ on the product space $K\times X$ is a bundle map on $K\times X$ (more precisely, see Def. \ref{on-cocycle}(i)), which can be extracted from the variational equation or the difference equation of various nonlinear evolution equations. Roughly speaking, a decomposition of quasi-all invariant bundles is that the product space $K\times X$ is decomposed as a whitney sum of several invariant finite-dimensional continuous subbundles and a positive invariant finite-codimensional continuous subbundle with respect to $(F,\mathcal{T})$ (see Def. \ref{DIB}). This property of a linear cocycle provides a fundamental base of qualitative analysis for various evolution systems, which is proved and obtained by us for the linear cocycle admitting a $k_i$-exponential separation with different $k_i\in \mathbb{N}^+$ at the same time (see Lemma \ref{Extending-uni} and the proof in Theorem A). 

In brief, a $k$-exponential separation of the linear cocycle $(F,\mathcal{T})$ on $K\times X$ is that the product space $K\times X$ is decomposed as a whitney sum of an invariant $k$-dimensional continuous subbundle $K\times (P_x)$ and a positive invariant $k$-codimensional continuous subbundle $K\times (Q_x)$ with respect to $(F,\mathcal{T})$ such that for any $x\in K$, the proportion of the evolution rate of vectors originating from $P_x\setminus\{0\}$ to the one of vectors originating from $Q_x\setminus\{0\}$ under the action of $(F,\mathcal{T})$ is controlled by an exponential function (more precisely, see Def. \ref{k-ES}). This is a very important property for dynamical analysis of differential equations and dynamical systems (See e.g. \cite{C-H,F-Wu,H-P-M,L-R-S,M-S-jde03,M-S-tams-13,M-S-jmaa-13,P-1,Pu,Pu-S,P-T,Shen-Yi,Tere}). Pol\'{a}$\check{c}$ik and Tere$\check{s}\check{c}\acute{a}$k (See \cite{P-T}) proved that a linear cocyle on the certain product space admits a $1$-exponential separation by utilizing strong positivity and compactness of fibre-maps with respect to a solid convex cone (see Def. \ref{k-cone} and the concepts related to a $k$-cone; see Remark \ref{sp-si}(i) for positivity and strong positivity with respect to a cone). Later on, Tere$\check{s}\check{c}\acute{a}$k (See \cite{Tere}) proved that a linear cocyle on the certain product space admits a $k$-exponential separation for a general positive integer $k$ under the assumption that fibre-maps of the linear cocycle are compact and strongly positive with respect to a complemented and $k$-solid cone (see Def. \ref{k-cone} and the concepts related to a $k$-cone). Since strong positivity of fibre-maps of a linear cocycle with respect to some $k$-cones is a very important feasible condition that can be verified in various systerms generated from evolution equations, we choose linear cocycles with a set of strongly invariant cones (see Def.\ref{on-cocycle} (iii)) as our research objects to investigate their property that's a decomposition of quasi-all invariant bundles.

In this paper, we will prove that the ranks $\{k_i\}_{i=0}^{N}$ of the strongly invariant cones with respect to the linear cocycle $(F,\mathcal{T})$ (see Def.\ref{on-cocycle} (iii)) indicates the index set $J$ of the decomposition of quasi-all invariant bundles for the linear cocycle $(F,\mathcal{T})$, which is used to symbolize the dimensions of positive invariant bundles with respect to $(F,\mathcal{T})$ that referred to by the decomposition of quasi-all invariant bundles for $(F,\mathcal{T})$. Our approach makes us need not to care about the location-relations among these strongly invariant cones with respect to $(F,\mathcal{T})$ and reveals that the ranks of these strongly invariant cones is the essentially visual charateristic to display the type (the index set $J$) of the decomposition of quasi-all invariant bundles, as well as gives us the advantage of flexible choice one by one on the possible cones strongly invariant with respect to $(F,\mathcal{T})$ in applications. The key of our approach is to analyze the containing-relations among positive invariant bundles referred to by the $k$-exponential separation and $k^{\prime}$-exponential separation of $(F,\mathcal{T})$ without all conditions about strongly invariant cones with respect to $(F,\mathcal{T})$. We prove that the invariant $k$-dimensional continuous bundle $K\times (P_x)$ {\rm(}resp. $k^{\prime}$-dimensional continuous bundle $K\times (P^{\prime}_x)$ {\rm)} and the positive invariant $k$-codimensional continuous bundle $K\times (Q_x)$ {\rm(}resp. $k^{\prime}$-codimensional continuous bundle $K\times (Q^{\prime}_x)$ {\rm)} referred to by the $k$-exponential separation {\rm(}resp. $k^{\prime}$-exponential separation{\rm)} of $(F,\mathcal{T})$ with $k\leq k^{\prime}$ have the following relations: $P_x \subset P^{\prime}_x$ and $Q^{\prime}_x\subset Q_x$ for any $x\in K$ (see Lemma \ref{Extending-uni}). When the linear cycocle $(F,\mathcal{T})$ is equipped with a set of strongly invariant cones, we can obtain the $k_i$-exponential separation of $(F,\mathcal{T})$ with different $k_i\in\mathbb{N}^+$ (the rank of the strongly invariant cone with respect to $(F,\mathcal{T})$) due to the result on $k$-exponential separation obtained by Tere$\check{s}\check{c}\acute{a}$k (See \cite{Tere}). As a series of consequences, we obtain the decomposition of quasi-all invariant bundles for $(F,\mathcal{T})$, as well as the certain location-relations between the fibres of positive invariant bundles and these cones, and the order-relations among characteristic exponents of $(F,\mathcal{T})$ on the positive invariant bundles (see Def. \ref{ki-CLyaE} and Theorem A). We put more results on different exponents of $(F,\mathcal{T})$ (see Def. \ref{k-LyaE-CoLyaE-LyaV} and Def. \ref{ki-CLyaE}) in Corollary B.

This paper is organized as follows. In section 2, we give some notations, definitions and preliminary lemmas. In section 3, we present our main results. In section 4, we prove our main results. In section 5, we apply our main results to a class of semilinear equations on a Hilbert space.

\section{Preliminary}

Let $X$ be a Banach space equipped with the norm $\norm{\cdot}$, and $X^*$ be its dual space. Let $K\subset X$ be a nonempty compact subset of $X$. Denoted by $L(X)$ the space of all linear bounded operators on $X$, equipped with the norm $\norm{\cdot}_{L(x)}$.
\begin{definition}\label{k-cone}A closed subset $C\subset X$ is called as {\it a cone of rank $k$ \rm{(}abbr. $k$-cone\rm{)}} if:

{\rm (i)} $\mathbb{R}\cdot C\subset C$; 

{\rm (ii)} $\max\{\text{dim}\,H: \,\,H\subset C\,\, \text{is a linear subspace} \}=k$. 
\end{definition} A $k$-cone $C$ is called as {\it solid} if its interior $\text{Int}C\neq \emptyset$, and is called as {\it k-solid} if there is a $k$-dimensional subspace $H$ such that $H\setminus\{0\}\subset \text{Int}C$. Moreover, it is called as {\it complemented} if there is a closed $k$-codimensional linear subspace $H^c$ such that $H^c\setminus \{0\}\subset X\setminus C$. We recall here that a nonempty closed set $X^+\subset X$ is called a {\it convex cone} if $\mathbb{R}\cdot X^+\subset X^+$, $x+y\in X^+$ for any $x,\,y\in X^+$ and $X^+\cap (-X^+)=\{0\}$. Furthermore, $X^+$ is called {\it soild} if the interior ${\rm Int} X^+$ of $X^+$ is nonempty. It is clear that $X^+\cup (-X^+)$ is a cone of rank 1.

Let $G(k,X)$ be {\it the Grassmanian of $k$-dimensional linear subspaces of $X$}, which consists of all $k$-dimensional linear subspace in $X$. $G(k, X)$ is a completed metric space by endowing {\it the gap metric} (see \cite{K,L-L}). More precisely, the gap metric is defined as

\begin{equation*} d(L_1,L_2)=\max\left\{\sup_{v\in L_1\cap S}\inf_{u\in L_2\cap S}\norm{v-u}, \sup_{v\in L_2\cap S}\inf_{u\in L_1\cap S}\norm{v-u}\right\},
 \end{equation*} for any nontrivial
closed subspaces $L_1,L_2\subset X$, where $S=\{v\in X:\norm{v}=1\}$ is the unit sphere. 

Given a $k$-cone $C$. Denoted by $G_k(C)$ the set of $k$-dimensional subspaces inside $C$ and $G_k({\rm Int} C)$ the set of $k$-dimensional subspaces inside ${\rm Int} C\cup \{0\}$ for $C$ being $k$-solid, respectively, i.e.,

$$G_k(C) =\{L\in G(k,X):\,L\subset C\}\quad \text{and} \quad G_k({\rm Int} C)=\{L\in G(k,X):\,L\setminus\{0\}\subset {\rm Int} C\}.$$ For $C$ being a complemented and $k$-solid cone, let 

$$G_k(C^*)=\{L\in G(k,\,X^*):\, {\rm Ker}(L)\cap C=\{0\}\}.$$ 

\vskip 3mm

Let $\{P_x\}_{x\in K}$ be a set of $k$-dimensional vector subspaces of $X$, and $\{Q_x\}_{x\in K}$ be a set of $k$-codimensional closed vector subspaces of $X$.
\begin{definition}
{\rm (i)} $K\times (P_x)$ is called {\it a $k$-dimensional continuous vector bundle} on $K$ {\rm(}for short, $k$-dimensional continuous bundle{\rm)} if the map $K\mapsto G(k,\,X): x\mapsto P_x$ is continuous;

{\rm (ii)} $K\times (Q_x)$ is called {\it a $k$-codimensional continuous vector bundle on $K$} {\rm(}for short, $k$-codimensional continuous bundle{\rm)} if there is a $k$-dimensional continuous vector bundle $K\times (L_x)\subset K\times X^*$ such that the kernel ${\rm Ker}(L_x)=Q_x$ for each $x\in K$. 
\end{definition}
 
Let $K\times (P_x)$ be a $k$-dimensional continuous bundle on $K$, and $K\times (Q_x)$ be a $k$-codimensional continuous bundle
on $K$ such that $X=P_x\oplus Q_x$ for all  $x\in K$. $P_x$ (resp. $Q_x$) is called {\it the fibre of $K\times(P_x)$ {\rm(}resp. $K\times (Q_x)${\rm)} at $x\in K$}. We define the {\it set of projections associated with the decomposition} $X=P_x\oplus Q_x$ as $\{\Pi^{P_x}_Q\}_{x\in K}$, where $\Pi^{P_x}_Q$ is the linear projection of $X$ on $P_x$ along $Q_x$ for each $x\in K$.
Write $\Pi^{Q_x}_P={\rm Id}-\Pi^{P_x}_Q$ for each $x\in K$. Clearly, $\Pi^{Q_x}_P$ is the linear projection of $X$ on $Q_x$ along $P_x$. Moreover,
both $\Pi^{P_x}_Q$ and $\Pi^{Q_x}_P$ are continuous with respect to $x\in K$.

\vskip 5mm
Throughout this paper, let $F$ be a homeomorphism on the nonempty compact set $K$, and $\mathcal{T}$ be a continuous map from $K$ to $L(X)$ defined by $\mathcal{T}(x)=T_x\in L(X)$. Thus, $\sup\limits_{x\in K}\norm{T_x}_{L(X)}$ is a fintie positive number. Denoted by $T_x^n=T_{F^{n-1}(x)}\circ T^{n-1}_{x}$ for any $n\in\mathbb{N}^+$, and $T^0_x={\rm Id}$ for any $x\in K$. Let 

$$\begin{aligned}&I=\{0,1,2,\cdots,N\}\,\, {\rm with\,\,an \,\,integer}\,\,N\geq1,\,\, {\rm and }\\
&J=\{j_i\in \mathbb{N}^+: i\in I\}\,\, {\rm with}\,\, k_{i}=\sum_{l=0}^{i}j_l \,\,{\rm for\,\, any}\,\, i\in I.\end{aligned}$$ Throughout this paper, let $\mathcal{C}=\{C_i\}_{i\in I}$ be a set of cones such that $C_i$ being a complemented and $k_i$-solid cone with $0<k_i<k_{i+1}$ for any $i\in I\setminus\{N\}$, and $k_i\in \mathbb{N}^+$ for any $i\in I$. The unit sphere is denoted by 

$$S=\{v\in X:\norm{v}=1\}.$$ Denoted by $\mathcal{SP}(C,\tilde{C})\subset L(X)$ the set of all linear operators $T$ such that $T(C\setminus\{0\})\subset {\rm Int} \tilde{C}$, where $C,\tilde{C}$ are two complemented and $k$-solid cones.

\begin{definition}\label{on-cocycle}

\noindent {\rm (i)} The linear cocycle of $F$ and $\mathcal{T}$ on $K\times X$ is the map

$$(F,\mathcal{T}): K\times X\rightarrow K\times X, \,\,(x,v)\mapsto (F(x), T_x v).$$ For short, $(F,\mathcal{T})$ on $K\times X$ is written as $(F,\mathcal{T})$. Here, $F$ and $\mathcal{T}$ are called the force-map and fibre-map-value map of $(F,\mathcal{T})$ respectively; $K$ and $X$ are called as the base space and the fibre space of $(F,\mathcal{T})$ respectively; $T_x$ is called as the fibre map of the linear cocycle $(F,\mathcal{T})$ at $x\in K$.

{\rm (ii)} The $k$-dimensional continuous bundle $K\times (P_x)$ {\rm (resp.} the $k$-codimensional continuous bundle $K\times (Q_x)${\rm)} is called as positive invariant with respect to $(F,\mathcal{T})$ if $TP_x\subset P_{F(x)}$ {\rm(resp.} $TQ_x\subset Q_{F(x)}${\rm)} for any $x\in K$; The $k$-dimensional continuous bundle $K\times (P_x)$ {\rm (resp.} the $k$-codimensional continuous bundle $K\times (Q_x)${\rm)} is called as invariant with respect to $(F,\mathcal{T})$ if $TP_x=P_{F(x)}$ {\rm(resp.} $TQ_x=Q_{F(x)}${\rm)} for any $x\in K$.

{\rm (iii)} The set of cones $\mathcal{C}$ is called as positve invariant {\rm (for short, invariant)} with respect to $(F,\mathcal{T})$ if $T_x C_i \subset C_i $ for any $(x,i)\in K\times I$; it is called as strongly positive invariant {\rm (for short, strongly invariant)} with respect to $(F,\mathcal{T})$ if $T_x (C_i\setminus\{0\})\subset {\rm Int} C_i$ for any $(x,i)\in K\times I$.
\end{definition}

\begin{rem}\label{sp-si}
{\rm (i)}The $k$-solid complemented cone $C$ is called {\rm(}resp., strongly positive invariant{\rm)} positive invariant w.r.t. $(F,\mathcal{T})$ if {\rm(}resp., $T_x (C\setminus\{0\})\subset{\rm Int}C$ {\rm)} $T_x C\subset C $ for any $x\in K$; for short, {\rm(}resp., strongly invariant{\rm)} invariant. $T_x$ is called strongly positive {\rm(}resp. positive{\rm)} with respect to a cone $C_{alg}$ if $T_x (C_{alg}\setminus\{0\})\subset{\rm Int}C_{alg}$ {\rm(}resp. $T_x C_{alg}\subset C_{alg} ${\rm)}, where $C_{alg}$ is the $k$-cone $C$ or a solid convex cone $X^+$.

{\rm (ii)}If $K=\{e\}$ is an equilibrium of $F$, the invariance of $\mathcal{C}$ w.r.t. $(F,\mathcal{T})$ implies that $T_e$ is positive w.r.t. $C_i$ {\rm(}i.e., $T_e C_i\subset C_i ${\rm)} for any $i\in I$; Similiarly, the strong invariance of $\mathcal{C}$ implies that $T_e$ is strongly positive w.r.t. $C_i$ {\rm(}i.e., $T_e (C_i\setminus \{0\})\subset {\rm Int} C_i ${\rm)} for any $i\in I$.
\end{rem}

\begin{definition}\label{k-ES}$(F,\mathcal{T})$ on $K\times X$ admits a $k$-exponential separation {\rm(}abbr. $k$-ES{\rm)} if

{\rm (i)} there are a $k$-dimensional continuous bundle $K\times (P_x)$ and a $k$-codimensional continuous bundle $K\times (Q_x)$ such that $X=P_x\oplus Q_x$ for any $x\in K$;

{\rm (ii)} $T_x P_x=P_{F(x)}$ and $T_x Q_x\subset Q_{F(x)}$ for any $x\in K$;

{\rm (iii)} there are constants $M>0$, $\gamma\in (0,1)$ such that 

\begin{equation}\label{Separation}\norm{T^n_x w}\leq M \gamma^n \norm{T^n_x v}\end{equation} for any $v\in P_x\cap S$, $w\in Q_x\cap S$, $x\in K$ and $n\in \mathbb{N}^+$. 

In additional, 

{\rm (iv)} if $C$ is a $k$-solid cone and $P_x\setminus \{0\}\subset {\rm Int} C$, $Q_x\setminus\{0\}\subset X\setminus C$ for any $x\in K$, then it is called that $(F,\mathcal{T})$ on $K\times X$ admits a $k$-exponential separation associated with $C$.

For short, we call that $(F,\mathcal{T})$ admits a $k$-exponential separation. We also call $(F,\mathcal{T})$ admits a $k$-exponential separation as a $k$-exponential separation of $(F,\mathcal{T})$ holds.

If $K\times (P_x)$ and $K\times (Q_x)$ are the unique $k$-dimensional continuous bundle and $k$-codimensional continuous bundle such that Def.\ref{k-ES}(i)-(iii) hold, then we call that $(F,\mathcal{T})$ on $K\times X$ has the unique $k$-exponential separation {\rm(}or, the unique $k$-exponential separation of $(F,\mathcal{T})$ holds.{\rm)}.
\end{definition}

\begin{definition}\label{k-LyaE-CoLyaE-LyaV}

{\rm (i)} Let $K\times (P_x)$ and $K\times (Q_x)$ are an invariant $k$-dimensional continuous bundle and a positive invariant  $k$-codimensional continuous bundle with respect to $(F,\mathcal{T})$ respectively such that $P_x\oplus Q_x=X$ for any $x\in K$. For any $x\in K$, we define the {\it $k$-Lyapunov exponent $\lambda_{\mathcal{T},kx}$ of $(F,\mathcal{T})$ on $K\times(P_x)$ {\rm(}for short, $k$-Lyapunov exponent $\lambda_{\mathcal{T},kx}$}) and the $k$-coLyapunov exponent $\lambda^{co}_{\mathcal{T},kx}$ of $(F,\mathcal{T})$ on $K\times (Q_x)$ {\rm(}for short, $k$-coLyapunov exponent $\lambda^{co}_{\mathcal{T},kx}${\rm)} as 

\begin{equation}\begin{aligned}&\quad\lambda_{\mathcal{T},kx}=\limsup\limits_{n\rightarrow +\infty}\frac{\log m(T_x^n\mid_{P_x})}{n}\\
&\quad\quad {\rm and}\\
&\lambda^{co}_{\mathcal{T},kx}=\limsup\limits_{n\rightarrow +\infty}\frac{\log\norm{T_x^n \mid_{Q_x} }_{L(X)}}{n} \end{aligned}\end{equation} respectively, where $m(T_x^n\mid_{P_x})=\inf\limits_{v\in P_x\cap S}\norm{T^n_x v}$;

{\rm (ii)} For any $x\in K$ and $v\in X\setminus\{0\}$, we can also define the Lyapunov exponent of $(F,\mathcal{T})$ {\rm(}for short,  Lyapunov exponent{\rm)} as

\begin{equation}\lambda_{\mathcal{T},x}(v)=\limsup\limits_{n\rightarrow +\infty}\frac{\log \norm{T_x^nv}}{n}.\end{equation}
\end{definition}

\begin{lemma}\label{k-Lya-G} Assume that $(F,\mathcal{T})$ admits a $k$-exponential separation, whose positive invariant bundles are $k$-dimensional continuous bundle $K\times (P_x)$ and $k$-codimensional continuous bundle $K\times (Q_x)$. Then, there is a $\gamma\in(0,1)$ such that

\begin{equation}\lambda_{\mathcal{T},kx}+\log(\gamma)\geq \lambda^{co}_{\mathcal{T},kx}\end{equation} for any $x\in K$.
\end{lemma}

\begin{proof} Since $K\times (P_x)$ is a $k$-dimensional continuous bundle, one has that $P_x\cap S$ is compact for any $x\in K$. Together with the property {\rm (iii)} of Def.$\,$\ref{k-ES} for the $k$-ES of $(F,\mathcal{T})$, there exist constants $M>0$ and $\gamma\in(0,1)$ such that for any $w\in Q_x\cap S$ and $x\in K$, 

\begin{equation}\label{H-ES-iii}\norm{T^n_x w}\leq M \gamma^n m(T^n_x\mid_{P_x}).\end{equation} It then follows that $\sup\limits_{w\in Q_x\cap S}\norm{T^n_x w}\leq M \gamma^n m(T^n_x\mid_{P_x})$ for any $x\in K$ and hence, $\lambda_{\mathcal{T},kx}+\log(\gamma)\geq \lambda^{co}_{\mathcal{T},kx}$ for any $x\in K$.
\end{proof}

\begin{definition}\label{DIB} $(F,\mathcal{T})$ on $K\times X$ admits a decomposition of quasi-all invariant bundles with the index set $J=\{j_i:\,i\in I\}\subset \mathbb{N}^+$ {\rm(}abbr. DIB{\rm)} if

{\rm (i)} there are $j_i$-dimensional continuous bundle $K\times (P_{i,x}),\,i\in I$ and $k_{N}$-codimensional continuous bundle $K\times (Q_x)$ such that $X=P_{0,x}\oplus P_{1,x}\oplus\cdots P_{N,x}\oplus Q_{x}$ for any $x\in K$;

{\rm (ii)} $T_x P_{i,x}=P_{i,F(x)}$ and $T_x Q_x\subset Q_{F(x)}$ for any $x\in K$ and $i\in I$.

For short, we call that $(F,\mathcal{T})$ admits a decomposition of quasi-all invariant bundles. We also call that $(F,\mathcal{T})$ admits a decomposition of quasi-all invariant bundles as a decomposition of quasi-all invariant bundles for $(F,\mathcal{T})$ holds.
\end{definition}

\begin{definition}\label{ki-CLyaE} Assume that $(F,\mathcal{T})$ admits a DIB w.r.t. $J$, whose positive invariant bundles are denoted by the same notations in Definition \ref{DIB}. For each $x\in K$ and $i\in I$, we define a $j_i$-charateristic exponent of $(F,\mathcal{T})$ on $K\times (P_{i,x})$ {\rm(}for short, $j_i$-charateristic exponent{\rm)} as 

\begin{equation} c_{\mathcal{T},ix}=\limsup\limits_{n\rightarrow +\infty}\frac{\log m(T_x^n\mid_{P_{i,x}})}{n},
\end{equation} where $m(T_x^n\mid_{P_{i,x}})=\inf\limits_{v\in P_{i,x}\cap S}\norm{T^n_x v}$.
\end{definition}

Assume that $(F,\mathcal{T})$ admits a decomposition of quasi-all invariant bundles with the index set $J$, whose related notations are the same as the ones in {\rm Def.} \ref{DIB}. Hereafter, for any $x\in K$,

{\rm (i)} denoted by 

$$\begin{aligned}&SP_{i,x}=P_{0,x}\oplus\cdots\oplus P_{i,x}\,\,{\rm with}\,\,i\in I,\\
&Q_{N,x}=Q_x,\quad Q_{i,x}=P_{i+1,x}\oplus\cdots\oplus P_{N,x}\oplus Q_x\,\,{\rm with} \,\,i\in I\setminus\{N\} \\
&\quad {\rm and}\\
&Q_{P^-_{0,x}}=Q_{0,x},\quad Q_{P^-_{i,x}}=SP_{i-1,x}\oplus Q_{i,x}\,\,{\rm with}\,\,i\in I\setminus\{0\};\end{aligned}$$ 

{\rm (ii)} denoted by $\lambda_{\mathcal{T},k_{i} x}$ the $k_{i}$-Lyapunov exponent of $(F,\mathcal{T})$ on $K\times (SP_{i,x})$, and $\lambda^{co}_{\mathcal{T},k_{i} x}$ the $k_{i}$-coLyapunov exponent of $(F,\mathcal{T})$ on $K\times (Q_{i,x})$ respectively for any $i\in I$. 

\section{Main results}

{\bf Theorem A.} Assume that the set of cones $\mathcal{C}$ is strongly invariant with respect to the linear cocycle $(F,\mathcal{T})$ on $K\times X$ such that $T_x$ is compact for any $x\in K$. Then, $(F,\mathcal{T})$ on $K\times X$ admits a decomposition of quasi-all invariant bundles with the index set $J$, whose positive invariant bundles are $j_i$-dimensional bundle $K\times (P_{i,x})$ and $k_N$-codimensional bundle $K\times (Q_x)$ such that:

{\rm (i)} $X=P_{0,x}\oplus\cdots\oplus P_{N,x}\oplus Q_x$ for any $x\in K$;

{\rm (ii)} $T_x P_{i,x}=P_{i,F(x)}$ and $T_x Q_{x}\subset Q_{F(x)}$ for any $x\in K$ and $i\in I$; 

{\rm (iii)}  $P_{0,x}\setminus \{0\}\subset {\rm Int}C_0$ and $P_{i+1,x}\setminus \{0\}\subset {\rm Int}C_{i+1}\setminus C_i$ for any $x\in K$ and $i\in I\setminus\{N\}$;

{\rm (iv)} $Q_{x}\setminus\{0\}\subset X\setminus C_N$;

{\rm (v)} there exists a $\gamma_i\in(0,1)$ such that $c_{\mathcal{T},ix}+\log(\gamma_i)\geq c_{\mathcal{T},i+1 x}$ for any $x\in K$ and $i\in I\setminus \{N\}$.

\vskip 3mm
\noindent {\bf Corollary B.} Assume that all hypothesises and notations in Theorem A hold. Then,

{\rm (i)} $c_{\mathcal{T},Nx}\geq \lambda_{\mathcal{T},k_Nx}$ and $c_{\mathcal{T},ix}\geq \lambda_{\mathcal{T},k_ix}\geq c_{\mathcal{T},i+1x}-\log(\gamma_i) $ for any $x\in K$ and $i\in I\setminus\{N\}$;

{\rm (ii)} $\lambda_{\mathcal{T},x}(v)=\lambda_{\mathcal{T},x}(\Pi^{P_{i,x}}_{Q_{P^-_i}}v)$ for any $v\in X\setminus Q_x$ and $x\in K$ if $i\in I$ is the smallest index such that $\Pi^{P_{i,x}}_{Q_{P^-_i}}v\neq0$.

\begin{rem} In usual,  it is difficult to know whether the linear cocycle $(F,\mathcal{T})$ on the whole product space $K\times X$ admits a decomposition of quasi-all invaraint bundles with the certain index set $J$ or not. Fortunately, to prove the strong invariance of a set of cones $\mathcal{C}$ w.r.t. the linear cocycle $(F,\mathcal{T})$ is a very important and feasible route to comfirm this property.
\end{rem}

\section{Proofs of main results}
Before the proofs, we give some technical lemmas.

\begin{lemma}\label{L: k-ES} Assume that the linear cocycle $(F,\mathcal{T})$ on $K\times X$ is such that $T_x$ is compact for any $x\in K$. Assume that $C$ is a strongly invariant  complemented and $k$-solid cone with respect to $(F,\mathcal{T})$ on $K\times X$. Then, $(F,\mathcal{T})$ on $K\times X$ admits a $k$-exponential separation associated with $C$.
\end{lemma}

\begin{proof} See \cite[Theorem 4.1]{Tere}. See also \cite[Lemma 2.7]{F}, \cite[Lemma 3.3]{F-II}, \cite[Proposition 3.2]{FWW-2} and \cite[Lemma 3.3]{FWW-3} for versions of the linear skew-product flow or semiflow.
\end{proof}

\begin{lemma}\label{L: dist-kbpc} Assume that the linear cocycle $(F,\mathcal{T})$ on $K\times X$ admits a $k$-exponential separation associated with $C$, whose positive invariant $k$-dimensional continuous bundle and $k$-codimensional continuous bundle are $K\times (P_x)$ and $K\times (Q_x)$ respectively. Then, there exists a constant $\delta_{P}>0$ such that 

$$\{v\in X: d_{usual}(v,P_x\cap S)\leq\delta_{P}\}\subset {\rm Int} C$$ for any $x\in K$, where $ d_{usual}(v,P_x\cap S)=\inf\limits_{w\in P_x\cap S}\{\norm{v-w}\}$.   
\end{lemma}

\begin{proof} See \cite[Lemma 2.8]{F}.
\end{proof}

\begin{lemma}\label{Extending-uni} Assume that the linear cocycle $(F,\mathcal{T})$ on $K\times X$ admits a $k$-exponential separation and a $k^{\prime}$-exponential separation such that 

{\rm (i)} $k\leq k^{\prime}$ with $k,k^{\prime}\in\mathbb{N}^+$;

{\rm (ii)} the invariant $k$-dimensional continuous bundle and $k^{\prime}$-dimensional continuous bundle with respect to $(F,\mathcal{T})$ on $K\times X$ are denoted by $K\times (P_x)$ and $K\times (P^{\prime}_x)$ respectively, and the positive invariant $k$-codimensional continuous bundle and $k^{\prime}$-codimensional continuous bundle with respect to $(F,\mathcal{T})$ on $K\times X$ are denoted by $K\times (Q_x)$ and $K\times (Q^{\prime}_x)$ respectively;

{\rm (iii)} there are $M,M^{\prime}>0$ and $\gamma,\gamma^{\prime}\in(0,1)$ such that 
$$\begin{aligned} &\norm{T^n_x v_{Q_x}}\leq M\gamma^n \norm{T^n_x v_{P_x}}\\
&\norm{T^n_x v_{Q^{\prime}_x}}\leq M^{\prime}(\gamma^{\prime})^n \norm{T^n_x v_{P^{\prime}_x}}
\end{aligned}$$ for any $ v_{P_x}\in P_x\cap S, v_{Q_x}\in Q_x\cap S,v_{P^{\prime}_x}\in P^{\prime}_x\cap S,v_{Q^{\prime}_x}\in Q^{\prime}_x\cap S$.

Then, 

$$P_x\subset P^{\prime}_x\quad {\rm and} \quad Q^{\prime}_x\subset Q_x$$ for any $x\in K$.
\end{lemma}

\begin{proof} We firstly assert that $P_x\setminus Q^{\prime}_x\neq\emptyset$ for any $x\in K$. Prove be contrary. Suppose that $P_z\subset Q^{\prime}_z$ for some $z\in K$. If $P^{\prime}_z\subset P_z$, then $P^{\prime}_z\subset P_z \subset Q^{\prime}_z$, a contradiction to $X=P^{\prime}_z\oplus Q^{\prime}_z$. Thus, there is 
$v_{P^{\prime}_z}\in (P^{\prime}_z\setminus P_z)\cap S$. Recall that $F$ is a homeomorphism on $K$. By the compactness of $K$, there is an increasing sequence $\{n_i\}_{i=1}^{+\infty}\subset\mathbb{N}^+$ such that $\lim\limits_{i\rightarrow +\infty}n_i=+\infty$ and $\lim\limits_{i\rightarrow +\infty}F^{-n_i}(z)=\tilde{z}\in K$. $T_xP_x=P_{F(x)}\,{\rm and}\,T_xP^{\prime}_x=P^{\prime}_{F(x)}$ for any $x\in K$ implies that there is a $v_{P^{\prime}_{F^{-n_i}(z) }}\in P^{\prime}_{F^{-n_i}(z) }\setminus P_{F^{-n_i}(z) } $ such that $T^{n_i} v_{P^{\prime}_{F^{-n_i}(z) }}=v_{P^{\prime}_z}$ for any $i\in\mathbb{N}^+$. Furthermore, 

$$0<\frac{\norm{\Pi^{Q_z}_Pv_{P^{\prime}_z}}}{\norm{\Pi^{P_z}_Qv_{P^{\prime}_z}}}\leq M\gamma^{n_i} \frac{\norm{\Pi^{Q_{ F^{-n_i}(z)}}_Pv_{P^{\prime}_{F^{-n_i}(z)}}}}{\norm{\Pi^{P_{ F^{-n_i}(z)}}_Q v_{P^{\prime}_{F^{-n_i}(z) }}}}$$ for any $i\in\mathbb{N}^+$. Together with the continuity of $K\times(P^{\prime}_x)$, $P_x\oplus Q_x=X$ and the compactness of $P^{\prime}_x\cap S$ for any $x\in K$, there is a subsequence of $\{n_i\}_{i=1}^{+\infty}$, also denoted by $\{n_i\}_{i=1}^{+\infty}$ without loss of generality, such that 
$\lim\limits_{i\rightarrow +\infty}\frac{v_{P^{\prime}_{F^{-n_i}(z)}}}{\norm{v_{P^{\prime}_{F^{-n_i}(z)}}}}=v_{P^{\prime}_{\tilde{z}}}\in (P^{\prime}_{\tilde{z}}\setminus P_{\tilde{z}})\cap S\cap Q_{\tilde{z}}$. It is clear that $P_{\tilde{z}}\subset Q^{\prime}_{\tilde{z}}$ because of supposing $P_z\subset Q^{\prime}_z$. {\bf (}If not, there is a $v\in P_{\tilde{z}}\cap S$ such that $\Pi_{Q^{\prime}}^{P^{\prime}_{\tilde{z}}}v\neq0$ and hence, there is a positive number $\varepsilon$ such that $\varepsilon<\frac{ \norm{  \Pi_{Q^{\prime}}^{P^{\prime}_{\tilde{z}}}v} }{\norm{\Pi_{P^{\prime}}^{Q^{\prime}_{\tilde{z}}}v}  }\leq+\infty$. It then follows from the continuity of $K\times (P_x)$, $K\times (P^{\prime}_x)$ and $K\times (Q^{\prime}_x)$ that there exist $v_{F^{-n_i}(z)}\in P_{F^{-n_i}(z)}\cap S$ such that 

$$\frac{\varepsilon}{2}<\frac{\norm{\Pi_{Q^{\prime}}^{P^{\prime}_{F^{-n_i}(z)} }v_{F^{-n_i}(z)}}}{\norm{\Pi_{P^{\prime}}^{Q^{\prime}_{F^{-n_i}(z)} }v_{F^{-n_i}(z)}}}$$ for any sufficiently large $i\in \mathbb{N}^+$, and hence, it follows from the compactness of $P_z\cap S$ and a $k^{\prime}$-exponential separation of $(F,\mathcal{T})$ that there exist points $v_{n_i}\in P_z\cap S$ for any sufficiently large $i\in \mathbb{N}^+$ such that $v_{n_i}=\frac{T^{n_i}_{F^{-n_i}(z)}v_{F^{-n_i}(z)}}{\norm{ T^{n_i}_{F^{-n_i}(z)}v_{F^{-n_i}(z)}}}$, 

$$\frac{\norm{\Pi^{Q^{\prime}_z}_{P^{\prime}}v_{n_i}}}{\norm{\Pi^{P^{\prime}_z}_{Q^{\prime}}v_{n_i}}}< \frac{2\cdot M^{\prime}\cdot(\gamma^{\prime})^{n_i}}{\varepsilon}$$ and there is a subsequence of $\{n_i\}$, denoted also by $\{n_i\}$ such that $\lim\limits_{i\rightarrow +\infty}v_{n_i}=v_v \in P^{\prime}_{z}\cap P_{z}\cap S$, a contradiction to $P_z\subset Q^{\prime}_z$ and $P^{\prime}_z\oplus Q^{\prime}_z=X$.{\bf)} It then follows from (iii) that there is $v_{P_{\tilde{z}}}\in P_{\tilde{z}}\cap S$ such that 

\begin{equation}\label{contra-iii}\frac{1}{M^{\prime} (\gamma^{\prime})^n} \norm{T^n_{\tilde{z}}v_{P_{\tilde{z}}}}\leq\norm{T^n_{\tilde{z}}v_{P^{\prime}_{\tilde{z}}}}\leq M\gamma^n \norm{T^n_{\tilde{z}} v_{P_{\tilde{z}}}}\end{equation} for any $n\in\mathbb{N}^+$, a contradiction. Therefore, we have proved this assertion.

Now, we prove that $P_x\subset P^{\prime}_x$ for any $x\in K$ by contrary. Suppose that there is a $z\in K$ such that $P_z\setminus P^{\prime}_z\neq\emptyset$. Then, there is $v_{P_z}\in (P_z\setminus P^{\prime}_z)\cap S$ such that $\Pi^{Q^{\prime}_z}_{P^{\prime}}v_{P_z}\neq 0$. By the compactness of $K$, $P_x\cap S$ for any $x\in K$, and the continuity of $K\times(P_x)$ and $T_x P_x=P_{F(x)},\,\,T_x P^{\prime}_x=P^{\prime}_{F(x)}$ for any $x\in K$, there is an increasing sequence $\{n_i\}_{i=1}^{+\infty}\subset\mathbb{N}^+$ and $v_{P_{F^{-n_i}(z)}}\in P_{F^{-n_i}(z)}\setminus P^{\prime}_{F^{-n_i}(z)}$ such that $\lim\limits_{i\rightarrow +\infty}n_i=+\infty$, $F^{-n_i}(z)=\tilde{z}\in K$, $T^{n_i}_{F^{-n_i}(z)}v_{P_{F^{-n_i}(z)} }=v_{P_z}$ and $ \lim\limits_{i\rightarrow +\infty}\frac{v_{P_{F^{-n_i}(z)}}}{\norm{v_{P_{F^{-n_i}(z)}}}}=v_{P_{\tilde{z}}}\in P_{\tilde{z}}\cap S$. It then follows from (iii) that 

$$\frac{ \norm{\Pi^{Q^{\prime}_z}_{P^{\prime}}v_{P_z}}}{\norm{\Pi^{P^{\prime}_z}_{Q^{\prime}}v_{P_z }}}\leq M^{\prime}(\gamma^{\prime})^{n_i} \frac{\norm{\Pi^{Q^{\prime}_{ F^{-n_i}(z)}}_{P^{\prime}}v_{P_{F^{-n_i}(z)}}}}{\norm{\Pi^{P^{\prime}_{ F^{-n_i}(z)}}_{Q^{\prime}}v_{P_{F^{-n_i}(z)}}}}$$ for any $i\in\mathbb{N}^+$, and hence, $v_{P_{\tilde{z}}}\in (P_{\tilde{z}}\setminus P^{\prime}_{\tilde{z}})\cap S\cap Q^{\prime}_{\tilde{z}}=P_{\tilde{z}}\cap S\cap Q^{\prime}_{\tilde{z}}$. Together with $P^{\prime}_{\tilde{z}}$ being $k^{\prime}(\geq k)$-dimensional space, there is a $v_{P^{\prime}_{\tilde{z}}}\in (P^{\prime}_{\tilde{z}}\setminus P_{\tilde{z}})\cap S$ such that $\Pi^{Q_{\tilde{z}}}_{P} v_{P^{\prime}_{\tilde{z}}}\neq 0$. By the compactness of $K$, $P_x\cap S\,\,P^{\prime}_x\cap S$ for any $x\in K$, and the continuity of $K\times(P_x),\,\,K\times(P^{\prime}_x),\,\,K\times (Q^{\prime}_x)$ and $T_x P_x=P_{F(x)},\,\,T_x P^{\prime}_x=P^{\prime}_{F(x)},\,\,T_x Q^{\prime}_x\subset Q^{\prime}_{F(x)}$ for any $x\in K$, there is an increasing sequence $\{m_i\}_{i=1}^{+\infty}$ with $v_{P_{F^{-m_i}(\tilde{z})}}\in (P_{F^{-m_i}(\tilde{z})}\cap Q^{\prime}_{F^{-m_i}(\tilde{z})})\setminus\{0\}$ and $v_{P^{\prime}_{F^{-m_i}(\tilde{z})}}\in P^{\prime}_{F^{-m_i}(\tilde{z})}\setminus P_{F^{-m_i}(\tilde{z})}$ such that $\lim\limits_{i\rightarrow +\infty}m_i=+\infty$, $\lim\limits_{i\rightarrow +\infty}F^{-m_i}(\tilde{z})=\tilde{\tilde{z}}\in K$, $T^{m_i}_{F^{-m_i}(\tilde{z}) }v_{P_{F^{-m_i}(\tilde{z})}}=v_{P_{\tilde{z}}}$, $T^{m_i}_{F^{-m_i}(\tilde{z}) }v_{P^{\prime}_{F^{-m_i}(\tilde{z})}}=v_{P^{\prime}_{\tilde{z}}} $ and $\lim\limits_{i\rightarrow +\infty}\frac{v_{P_{F^{-m_i}(\tilde{z})}}}{\norm{v_{P_{F^{-m_i}(\tilde{z})}}}}=v_{P_{\tilde{\tilde{z}}}}\in P_{\tilde{\tilde{z}}}\cap S\cap Q^{\prime}_{\tilde{\tilde{z}}}$, $\lim\limits_{i\rightarrow +\infty}\frac{v_{P^{\prime}_{F^{-m_i}(\tilde{z})}}}{\norm{v_{P^{\prime}_{F^{-m_i}(\tilde{z})}}}}=v_{P^{\prime}_{\tilde{\tilde{z}}}}\in P^{\prime}_{\tilde{\tilde{z}}}\cap S$. By utilizing (iii) again, one has that 

$$\frac{\norm{\Pi^{Q_{\tilde{z}}}_{P}v_{P^{\prime}_{\tilde{z}}}}}{\norm{\Pi^{P_{\tilde{z}}}_{Q}v_{P^{\prime}_{\tilde{z}}}}}\leq M(\gamma)^{n_i} \frac{\norm{\Pi^{Q_{ F^{-n_i}(\tilde{z})}}_{P}v_{P_{F^{-n_i}(\tilde{z})}}}}{\norm{\Pi^{P_{ F^{-n_i}(\tilde{z})}}_{Q}v_{P_{F^{-n_i}(\tilde{z})}}}}$$ for any $i\in\mathbb{N}^+$, and hence, $v_{P^{\prime}_{\tilde{\tilde{z}}}}\in (P^{\prime}_{\tilde{\tilde{z}}}\setminus P_{\tilde{\tilde{z}}})\cap S\cap Q_{\tilde{\tilde{z}}}=P^{\prime}_{\tilde{\tilde{z}}}\cap S\cap Q_{\tilde{\tilde{z}}}$. Together with $v_{P_{\tilde{\tilde{z}}}}\in P_{\tilde{\tilde{z}}}\cap S\cap Q^{\prime}_{\tilde{\tilde{z}}} $ and (iii), one can obtain that (\ref{contra-iii}) holds by replacing $v_{P^{\prime}_{\tilde{z}}}, v_{P_{\tilde{z}}}$ as $v_{P^{\prime}_{\tilde{\tilde{z}}}}, v_{P_{\tilde{\tilde{z}}}}$ respectively. It is a contradiction. Thus, we have proved that $P_x\subset P^{\prime}_x$ for any $x\in K$.

For any given $v_{Q^{\prime}_x}\in Q^{\prime}_x\cap S$ with $x\in K$, it is clear that $\Pi^{Q_x}_{P}v_{Q^{\prime}_x}\neq 0$ because of $P_x\subset P^{\prime}_x$ and $P^{\prime}_x\cap Q^{\prime}_x=\{0\}$. Suppose that $\Pi^{P_x}_{Q}v_{Q^{\prime}_x}\neq 0$. It then follows from (iii) that

\begin{equation}\label{Q'-P-decom-var}\frac{\norm{\Pi^{Q_{F^{n}(x)}}_{P}T^{n}_xv_{Q^{\prime}_x}}}{\norm{\Pi^{P_{F^{n}(x)}}_{Q}T^{n}_x v_{Q^{\prime}_x}}}\leq M\gamma^{n} \cdot \frac{ \norm{\Pi^{Q_x}_{P}v_{Q^{\prime}_x}}}{\norm{\Pi^{P_x}_{Q}v_{Q^{\prime}_x}}}\end{equation} for any $n\in \mathbb{N}^+$. By utilizing the compactness of $K$, $P_x\cap S$ for any $x\in K$, and the continuity of $K\times(P_y),\,\,K\times (Q^{\prime}_y)$ such that $T_y P_y=P_{F(y)},\,\,T_y Q^{\prime}_y\subset Q^{\prime}_{F(y)}$ for any $y\in K$, there is an increasing sequence $\{n_i\}_{i\in\mathbb{N}^+}\subset \mathbb{N}^+$ such that $\lim\limits_{i\rightarrow +\infty}F^{n_i}(x)=\tilde{x}$ and 

$$\lim\limits_{i\rightarrow +\infty} \frac{ T^{n_i}_x v_{Q^{\prime}_x}}{\norm{T^{n_i}_x v_{Q^{\prime}_x}}}=v_{Q^{\prime}_{\tilde{x}}}\in P_{\tilde{x}}\cap S\cap Q^{\prime}_{\tilde{x}},$$  a contradition to $X=P^{\prime}_{\tilde{x}}\oplus Q^{\prime}_{\tilde{x}}$ and $P_{\tilde{x}}\subset P^{\prime}_{\tilde{x}}$. Thus, $\Pi^{P_x}_{Q}v_{Q^{\prime}_x}=0$ and $v_{Q^{\prime}_x}\in Q_x$. Then, $Q^{\prime}_x\subset Q_x$ for any $x\in K$.

Therefore, we have completed the proof.
\end{proof}  

\begin{rem}\label{uni} For the case $k=k^{\prime}$, it follows from Lemma \ref{Extending-uni} that there are an unique $k$-dimensional continuous bundle and an unique $k$-codimensional continuous bundle such that $(F,\mathcal{T})$ on $K\times X$ admits a $k$-ES.
\end{rem}

Now, we prove Theorem A.
\begin{proof} The strong invariance of $\mathcal{C}$ w.r.t. $(F,\mathcal{T})$ implies the strong invariance of $C_i$ w.r.t. $(F,\mathcal{T})$ for any $i\in I$.  
Togther with the compactness of $T_x,x\in K$ and Lemma \ref{L: k-ES}, one has that $(F,\mathcal{T})$ on $K\times X$ admits a $k_i$-ES associated with $C_i$ for any $i\in I$, whose invariant $k_i$-dimensional bundle and $k_i$-codimensional bundle are denoted by $K\times (SP_{i,x})$ and $K\times (Q_{i,x})$ respectively.

By virtue of Lemma \ref{Extending-uni} and $k_i<k_{i+1}$ for any $i\in I\setminus\{N\}$, ones have that $SP_{i,x}\subsetneq SP_{i+1,x}$ and $Q_{i+1,x}\subsetneq Q_{i,x}$ for any $x\in K$ and $i\in I\setminus\{N\}$. For any $x\in K$, let 

\begin{equation}\label{PQ0N}P_{0,x}=SP_{0,x}\subset {\rm Int} C_0\cup\{0\},\,\,Q_x=Q_{N,x}\subset (X\setminus C_{N})\cup\{0\}\,\, {\rm and}\,\, j_0=k_0.\end{equation} Then, $K\times (P_{0,x})$ is a $j_0$-dimensional bundle and $K\times (Q_x)$ is a $k_N$-codimensional bundle such that 

\begin{equation}\label{PQ0N-In}T_x P_{0,x}=P_{0,F(x)}\subset {\rm Int} C_0\cup\{0\}\,\, {\rm and}\,\,T_x Q_x\subset Q_{F(x)}\subset (X\setminus C_{N})\cup\{0\}\end{equation} for any $x\in K$, respectively. Let 

$$P_{i+1,x}=\Pi^{Q_{i,x}}_{ SP_{i}}SP_{i+1,x}$$ for any $x\in K$ and $i\in I\setminus\{N\}$. Since $(F,\mathcal{T})$ on $K\times X$ admits a $k_i$-ES associated with $C_i$ for any $i\in I\setminus \{N\}$, one has that 

\begin{equation}\label{Decom-to-P}\begin{aligned} &SP_{i+1,x}=\Pi^{SP_{i,x}}_{ Q_{i}}SP_{i+1,x}\oplus \Pi^{Q_{i,x}}_{ SP_{i}}SP_{i+1,x} =SP_{i,x}\oplus P_{i+1,x}\\
& {\rm and}\quad P_{i+1,x}\subset Q_{i,x} 
\end{aligned}\end{equation} for any $x\in K$ and $i\in I\setminus \{N\}$. Furthermore, ones have $T_x SP_{i,x}=SP_{i,F(x)}\subset {\rm Int} C_{i}\cup\{0\}$ and $T_x Q_{i,x}\subset Q_{i,F(x)}\subset (X\setminus C_{i})\cup \{0\}$ for any $x\in K$ and $i\in I$.  It then follows that 

\begin{equation}\label{P-Posi}\begin{aligned}&P_{i+1,x}\setminus \{0\}\subset {\rm Int} C_{i+1,x}\setminus C_{i,x}\\
& {\rm and} \quad T_x P_{i+1,x} \subset P_{i+1,F(x)}
\end{aligned}\end{equation} for any $x\in K$ and $i\in I\setminus \{N\}$. Note that $K\times (SP_{i,x})$ is a $k_i$-dimensional bundle for any $i\in I$. Then, for any $i\in I\setminus \{N\}$, 

\begin{equation}\label{De-J}j_{i+1}=k_{i+1}-k_i,\end{equation} and hence, $K\times (P_{i+1,x})$ is a $j_{i+1}$-dimensional bundle.

Now, we assert that for any $x\in K$ and $i\in I$, 

\begin{equation}\label{P-to-FP}T_x P_{i,x} = P_{i,F(x)}\end{equation} Since $P_{0,x}=SP_{0,x}$ and $T_x SP_{0,x}=SP_{0,F(x)}$ for any $x\in K$, it suffices to prove that $T_x P_{i+1,x}=P_{i+1,F(x)}$ for any $x\in K$ and $i\in I\setminus \{N\}$. Prove by contrary. Suppose $P_{i+1,F(x)}\setminus T_x P_{i+1,x}\neq \emptyset$ for a given $x\in K$ and $i\in I\setminus \{N\}$. Let $v_1\in P_{i+1,F(x)}\setminus T_x P_{i+1,x}$. It then follows (\ref{Decom-to-P}) and $T_x SP_{i+1,x}=SP_{i+1,F(x)}$ that there is a $v\in SP_{i+1,x}$ such that $T_x v=v_1$ and $\Pi^{SP_{i,x}}_{Q_i}v\neq 0$. By the unique $k_j$-ES of $(F,\mathcal{T})$ for any $j\in I$ because of Remark \ref{uni}, one has that $\Pi^{SP_{i,F(x)}}_{ Q_{i}}v_1=T_x \Pi^{SP_{i,x}}_{ Q_{i}}v\neq 0$, a contradiction to $v_1\in P_{i+1,F(x)}$ and $SP_{i+1,F(x)}=SP_{i,F(x)}\oplus P_{i+1,F(x)}$. Therefore, we have proved the assertion.

By virtue of (\ref{PQ0N}) and (\ref{Decom-to-P}), one has $SP_{N,x}=P_{0,x}\oplus P_{1,x}\oplus\cdots\oplus P_{N,x}$ for any $x\in K$. Note that $X=SP_{N,x}\oplus Q_x$ for any $x\in K$ because of the $k_N$-ES of $(F,\mathcal{T})$ holding. One has that 

$$X=P_{0,x}\oplus P_{1,x}\oplus\cdots\oplus P_{N,x}\oplus Q_x$$ for any $x\in K$. Together with (\ref{PQ0N}), (\ref{PQ0N-In}), (\ref{De-J}) and (\ref{P-to-FP}), $(F,\mathcal{T})$ admits the decomposition of invariant bundles with the index set $J$, whose invariant bundles are $j_i$-dimensional bundle $K\times (P_{i,x})$, $i\in I$ and $k_N$-codimensional bundle $K\times (Q_x)$. Futhermore, combinating with (\ref{P-Posi}), properties (i)-(iv) in Theorem A hold.

Finally, we prove the property {\rm (v)} in Theorem A. By the $k_i$-ES of $(F,\mathcal{T})$ for any $i\in I$, there are constants $M_i>0$ and $\gamma_i\in(0,1)$ such that 

\begin{equation}\label{gamma-for-kiES} \norm{T^n_x w}\leq M_i \gamma_i^n \norm{T^n_x v}\end{equation} for any $v\in SP_{i,x}\cap S$, $w\in Q_{i,x}\cap S$ and $x\in K$. It then follows from Lemma \ref{k-Lya-G} that $\lambda_{\mathcal{T},k_i x}+\log(\gamma_i)\geq \lambda_{\mathcal{T},k_i x}^{Co}$ for any $x\in K$ and $i\in I$, where $\lambda_{\mathcal{T},k_ix}$ and $\lambda^{co}_{\mathcal{T},k_ix}$ are $k_i$-Lyapunov exponent on $K\times (SP_{i,x})$ and $k_i$-coLyapunov exponent on $K\times (Q_{i,x})$ for $x\in K$ respectively. This implies  $c_{\mathcal{T},ix}+\log(\gamma_i)\geq c_{\mathcal{T},i+1 x}$ for any $x\in K$ and $i\in I\setminus \{N\}$. Thus, the property {\rm (v)} in Theorem A has been proved. 

Therefore, we have completed the proof of Theorem A.
\end{proof}

\vskip 5mm
{\rm Proofs of Corollary B:}
{\rm (i)} Clearly, $\lambda_{\mathcal{T}, k_0x}=c_{\mathcal{T}, 0x}$ for any $x\in K$. By Definition \ref{k-LyaE-CoLyaE-LyaV}(i) and \ref{ki-CLyaE}, one has that $c_{\mathcal{T}, ix}\geq \lambda_{\mathcal{T}, k_i x}$ for any $x\in K$ and $i\in I$. Let $\gamma_i\in(0,1),\,i\in I$ be the number in Definition \ref{k-ES} (iii) for $k_i$-ES of $(F,\mathcal{T})$. Together with Definition \ref{k-LyaE-CoLyaE-LyaV}(i) and Lemma \ref{k-Lya-G}, one has that 

$$c_{\mathcal{T}, ix}\geq\lambda_{\mathcal{T}, k_ix}\geq \lambda^{co}_{\mathcal{T}, k_ix}-\log(\gamma_i)\geq c_{\mathcal{T}, i+1x}-\log(\gamma_i)$$ for any $x\in K$ and $i\in I\setminus\{N\}$. Thus, we have completed the proof of (i).

{\rm (ii)} It is not difficult to see that 

\begin{equation}\label{vec-decom}v=\Pi^{Q_x}_{SP_N}v+\sum_{i=0}^{N}\Pi^{P_{i,x}}_{Q_{P_i^-} }v\end{equation} for any $v\in X$ and $x\in K$. Given $x\in K$ and $v\in X\setminus Q_x$ with the smallest index $i\in I$ such that $\Pi^{P_{i,x}}_{Q_{P_i^-} }v\neq 0$. By Theorem A, $(F,\mathcal{T})$ on $K\times X$ admits a $k_i$-ES, whose invariant bundles are $k_i$-dimensional bundle $K\times (SP_{i,x})$ and $k_i$-codimensional bundle $K\times (Q_{i,x})$ such that for any $w\in SP_{i,x}\cap S$, $v\in Q_{i,x}\cap S$ and $n\in\mathbb{N}^+$, $\norm{G^n_x w}\leq M_i \gamma_i^n \norm{G^n_x v}.$ Thus, 

\begin{equation}\norm{T^n_x \Pi^{Q_{i,x}}_{SP_i}v }\leq M_{i}\gamma^n_i\cdot\frac{\norm{\Pi^{Q_{i,x}}_{S\tilde{P}_i}v}}{\norm{\Pi_{Q_{i}}^{SP_{i,x}}v}}\cdot \norm{T^n_x \Pi_{Q_{i}}^{SP_{i,x}}v}\end{equation} for any $n\in \mathbb{N}^+$. Furthermore, ones have that

$$ \begin{aligned}&\lambda_{\mathcal{T},x}(v)=\limsup\limits_{n\rightarrow+\infty}\frac{\log\norm{T_x^{n}v} }{n}\leq \limsup\limits_{n\rightarrow +\infty}\frac{\log(\norm{T_x^n\Pi^{SP_{i,x}}_{Q_{i}}v}+\norm{ T_x^n\Pi^{Q_{i,x}}_{SP_i}v} ) }{n}\\
\leq&\limsup\limits_{n\rightarrow+\infty}\frac{\log(\norm{T_x^n\Pi^{SP_{i,x}}_{Q_{i}}v}(1+M_i\gamma_i^n\cdot \frac{\norm{\Pi^{Q_{i,x}}_{SP_i}v}}{\norm{\Pi^{SP_{i,x}}_{Q_{i}}v}}))}{n}=\lambda_{\mathcal{T},x}(\Pi^{SP_{i,x}}_{Q_{i}}v)\\
\quad&{\rm and}\\
&\lambda_{\mathcal{T},x}(v)=\limsup\limits_{n\rightarrow+\infty}\frac{\log\norm{T_x^{n}v} }{n}\geq \limsup\limits_{n\rightarrow +\infty}\frac{\log(\norm{T_x^n\Pi^{SP_{i,x}}_{Q_{i}}v}-\norm{ T_x^n\Pi^{Q_{i,x}}_{SP_i}v} ) }{n}\\
\geq&\limsup\limits_{n\rightarrow+\infty}\frac{\log(\norm{T_x^n\Pi^{SP_{i,x}}_{Q_{i}}v}(1-M_i\gamma_i^n\cdot \frac{\norm{\Pi^{Q_{i,x}}_{SP_i}v}}{\norm{\Pi^{SP_{i,x}}_{Q_{i}}v}}))}{n}=\lambda_{\mathcal{T},x}(\Pi^{SP_{i,x}}_{Q_{i}}v)
\end{aligned}$$ It yields that $\lambda_{\mathcal{T},x}(v)=\lambda_{\mathcal{T},x}(\Pi^{SP_{i,x}}_{Q_{i}}v)$. Recall that $i\in I$ is the smallest index $i\in I$ such that $\Pi^{P_{i,x}}_{Q_{P_i^-} }v\neq 0$. Then, $ \Pi^{SP_{i,x}}_{Q_{i}}v=\Pi^{P_{i,x}}_{Q_{P^-_i}}v$ and $\lambda_{\mathcal{T},x}(v)=\lambda_{\mathcal{T},x}(\Pi^{P_{i,x}}_{Q_{P^-_i}}v) $. Thus, we have proved (ii) of this corollary.
  
Therefore, we have completed the proof.

$\quad\quad\quad\quad\quad\quad\quad\quad\quad\quad\quad\quad\quad\quad\quad\quad\quad\quad\quad
\quad\quad\quad\quad\quad\quad\quad\quad\quad\quad\quad\quad\quad\quad\quad\quad\quad\quad\quad\quad\square$

\section{The decomposition of quasi-all invariant bundles for the semilinear equation on a Hilbert space}

\indent Let $X$ be a Hilbert space equiped with inner product $<\cdot\,\,,\,\,\cdot>$, which deduces a norm $\norm{\cdot}$. Let $A$ be a symmetric closed dense operator on $X$ such that 

{\rm (i)} the spectrum $\sigma(A)=\{\lambda_n\}_{n=1}^{+\infty}\subset \mathbb{R}$ is contained in its point spectrum such that $\lambda_{n+1}>\lambda_{n}$ for any $n\in\mathbb{N}^+$, and $\lim\limits_{n\rightarrow +\infty}\lambda_{n+1}-\lambda_n=+\infty$;

{\rm (ii)} the unit eigenfunctions related to $\lambda_n$ are $\psi_n^j$, $j\in\mathbb{N}^+,\,1\leq j\leq N_{n}$;

{\rm (iii)} $\cup_{n=1}^{\infty}\{\psi_n^j: j\in\mathbb{N}^+,\,\,1\leq j\leq N_{n}\}$ is an orthogonal basis of $X$.

We investigate the equation

\begin{equation}\label{Equ:semi}\frac{d u}{d t}+Au=\mathcal{F}(u), \,t>0,
\end{equation} which generates a dissipative and compact $\Phi_t$ semiflow on $X^{\alpha}$ with its global compact attractor $K\subset \mathcal{D}(A)\subset X^{\alpha}\subset X$. Here, $X^{\alpha}$ is the fractional power space $\mathcal{D}(A^{\alpha})=X^{\alpha}$ with a certain $\alpha\in(0,1)$. (For example, $X=L^2(S^1)$, where $S^1=\mathbb{R}/\mathbb{Z}$. $A=-\frac{\partial^2}{\partial x^2}$ with its domain $\mathcal{D}(A)=H^2(S^1)$ and its spectrum $\sigma(A)=\{4\pi^2n^2:\,n\in\mathbb{N}\}$. $X^{\frac{1}{2}}=H^1(S^1)$. $\mathcal{F}(u)=f(u(\cdot))$ for any $u\in X$, where $f\in C^2(\mathbb{R})$. There is a $\tilde{M}>0$ such that $x\cdot f(x)<0$ for any $\abs{x}>\tilde{M}$.) 
Assume that 

$${\rm (*)} \quad \mathcal{F}\in C^1(X)\,\, {\rm such\,\, that}\,\, \sup\limits_{u\in K}\norm{\mathcal{F}_u}_{L(X)}\leq M\,\, {\rm for\,\, some}\,\, M>0,$$ where $\mathcal{F}_u$ is the Fr\'{e}chet derivative of $\mathcal{F}$ at $u\in K$. 

By the knowledge in \cite[Chapter 2]{Pazy}, the semigroup $T(t)$ generated by $-A$ on $X$ is analytic. Hereafter, it is reasonable to assume that $T(t)$ is compact for $t>\frac{1}{2}$ by the knowlwdge in \cite[Chapter 6, 7]{Pazy}. Let $L_i={\rm Span}_{X}\{ \cup_{n=1}^{i}\{\psi_n^j: j\in\mathbb{N}^+,\,\,1\leq j\leq N_{n}\}\}$ and $L_i^c={\rm Span}_{X}\{ \cup_{n=i+1}^{+\infty}\{\psi_n^j: j\in\mathbb{N}^+,\,\,1\leq j\leq N_{n}\}\}$ for any $i\in\mathbb{N}^+$, where ${\rm Span}_{X}$ represents to take the closed linear subspace of $X$ spaned by vectors in one subset of $X$. Clearly, $L_i\oplus L_i^c=X$ and $T(t)L_i=L_i$, $T(t)L_i^c\subset L_i^c$ for any $i\in\mathbb{N}^+$. Then, ones have that 

\begin{equation}\label{Dich}\norm{T(t) \mid_{L_i}}_{L(X)}\geq e^{-\lambda_i\cdot t}\quad{\rm and}\quad\norm{T(t) \mid_{L_i^c}}_{L(X)}\leq e^{-\lambda_{i+1}\cdot t} \end{equation} for any $t>0$. Let  $P_i$ be the projection on $L_i$ along $L_i^c$, and $Q_i={\rm Id}-P_i$ for any $i\in \mathbb{N}^+$. Clearly, $\norm{P_i}_{L(X)},\,\norm{Q_i}_{L(X)}\leq 1$ for any $i\in \mathbb{N}^+$. Let 

$$\tilde{C}_i=\{u\in X:\,\norm{P_i u}\geq \norm{Q_i u}\} $$ and denote $\tilde{k}_i=\sum\limits_{n=1}^{i}N_i$ the dimension of $L_i$ for any $i\in\mathbb{N}^+$.

\begin{lemma}\label{cone-tilde} $\tilde{C}_i$ is a complemented and $\tilde{k}_i$-solid cone in $X$ for any $i\in\mathbb{N}^+$.
\end{lemma}

\begin{proof} Clearly, $\mathbb{R}\cdot \tilde{C}_i\subset \tilde{C}_i$ and $\tilde{C}_i$ is a closed subset of $X$. Note that $L_i\subset  \mathcal{D}(A)$ for any $i\in\mathbb{N}^+$. By the definition of $\tilde{C}_i$, ones have that $L_i\setminus\{0\}\subset {\rm Int}_{X} \tilde{C}_i=\{u\in X:\,\norm{P_i u}>\norm{Q_i u}\}$ and $L_i^c\setminus \{0\}\subset X\setminus \tilde{C}_i$, where ${\rm Int}_{X}$ represents to take the interior of a subset in $X$. Thus, $\tilde{C}_i$ is a complemented and $k_i$-solid cone in $X$.
\end{proof}

Let $u(t;\phi)=\Phi_t(\phi)$ (for simplicity, denoted by $u(t)$) for any $t\geq 0$. For any given $\phi\in K$, the variational equation of (\ref{Equ:semi}) along $u(t)$ is the following:

\begin{equation}\label{Equ:Varia-Semi} \frac{d v}{d t}+A v=\mathcal{F}_{u(t)} v,\,\,t>0.
\end{equation} Let $T_{\phi}^t$ be the solution map of (\ref{Equ:Varia-Semi}) defined by 

$$T^t_{\phi}\xi=v_{\phi}(t;\xi),\,\,t\geq 0,$$ where $v_{\phi}(\xi)$ is the unique solution of (\ref{Equ:Varia-Semi}) with initial data $\xi\in X$. It then follows from \cite[Lemma 3.3.2]{Dan} that 

\begin{equation}\label{const-varia-formu} T^t_{\phi}\xi=T(t)\xi+\int_0^t T(t-s) F_{u(s)} v_{\phi}(s;\xi)ds \in X\,\,{\rm for \,\, any}\,\,t> 0.
\end{equation}

\begin{lemma}\label{Sp-cone-tilde} There is a $N\in\mathbb{N}^+$ such that $T^t_{\phi}(\tilde{C}_{i}\setminus\{0\})\subset {\rm Int}_{X} \tilde{C}_i,\,t>0$ for any $\phi\in K$ and $i\in\mathbb{N}^+$ with $i\geq N$. 
\end{lemma}

\begin{proof} By calculation, there is a $\mu\in(0,1)$ such that 

\begin{equation}\label{Ineq}\lambda_{i+1}- \lambda_{i}-M(2+\frac{1}{1-\mu}+1+\mu)>0\end{equation} for any $i\in \mathbb{N}^+$ such that $\frac{\lambda_{i+1}- \lambda_{i}}{4 M}>1$. It then follows from $\lim\limits_{i\rightarrow +\infty}\lambda_{i+1}-\lambda_{i}=+\infty$ that there is $N\in\mathbb{N}^+$ such that (\ref{Ineq}) is solvable in $(0,1)$ for any $i\in\mathbb{N}^+$ with $i\geq N$. 

For any $i\in\mathbb{N}^+$ with $i\geq N$ and $\phi\in K$, given a $\xi\in\partial_{X} \tilde{C}_{i} \setminus\{0\}$, i.e., $\norm{P_i \xi}=\norm{Q_i \xi}$ and $\xi\in X\setminus \{0\}$. Clearly, there is a $\mu\in(0,1)$ such that (\ref{Ineq}) holds. By the continuity of solution $v_{\phi}(\xi)$,  there is a $t_{\phi}(\xi)>0$ such that  

\begin{equation}\label{band-boundary}(1-\mu)\norm{Q_i v_{\phi}(s;\xi)}\leq\norm{P_i v_{\phi}(s;\xi)}\leq (1+\mu)\norm{Q_i v_{\phi}(s;\xi)}\end{equation} for any $s\in [0,t_{\phi}(\xi)]$. Together with (\ref{const-varia-formu}), we have the following estimates:

\begin{equation}\label{cone-estimate-1}\begin{aligned}\norm{Q_i v_{\phi}(t;\xi)}&\leq e^{-\lambda_{i+1} t}\norm{Q_i\xi}+M\int_0^t e^{-\lambda_{i+1} \cdot(t-s)}(\norm{P_i v_{\phi}(s;\xi)}+\norm{Q_i v_{\phi}(s;\xi)})ds\\
&\leq e^{-\lambda_{i+1} t}\cdot [\norm{Q_i\xi}+M(2+u)\int_0^t e^{ \lambda_{i+1} \cdot s}\norm{Q_i v_{\phi}(s;\xi)}]ds]\\
{\rm and}\quad\quad&\\
\norm{P_i \xi}\quad\quad&\leq e^{\lambda_i\cdot t} \norm{P_i v_{\phi}(t;\xi)} +M\int_0^t e^{\lambda_i \cdot(t-s)}(\norm{P_i v_{\phi}(t-s;\xi)}+\norm{Q_i v_{\phi}(t-s;\xi)})ds\\
&\leq e^{\lambda_i\cdot t} \cdot[  \norm{P_i v_{\phi}(t;\xi)} +M(1+\frac{1}{1-\mu})\int_0^t e^{-\lambda_i \cdot s}\norm{P_i v_{\phi}(t-s;\xi)} ds]
\end{aligned}\end{equation} for any $t\in [0,t_{\phi}(\xi)]$. It then follows from Gronwall's inequality that for any $t\in [0,t_{\phi}(\xi)]$,

\begin{equation}\label{cone-estimate-2}\begin{aligned} &\norm{Q_i v_{\phi}(t;\xi)}\leq \norm{Q_i\xi}e^{(-\lambda_{i+1}+M(2+\mu))\cdot t}\\
&{\rm and}\\
 &\norm{P_i v_{\phi}(t;\xi)}\geq \norm{P_i\xi}e^{(-\lambda_{i}-M(1+\frac{1}{1-\mu}))\cdot t}
\end{aligned}\end{equation} Thus,

\begin{equation}\label{cone-estimate-3}\begin{aligned} \frac{\norm{Q_i v_{\phi}(t;\xi)}}{\norm{P_i v_{\phi}(t;\xi)}}&\leq e^{[M(2+\frac{1}{1-\mu}+1+\mu)-\lambda_{i+1}+\lambda_{i}]\cdot t} \cdot \frac{\norm{Q_i \xi}}{\norm{P_i\xi}}\\
&\overset{(\ref{Ineq}) {\rm \,holds\, for\, the\, certain\,\mu} }{<} \frac{\norm{Q_i \xi}}{\norm{P_i\xi}}\leq 1
\end{aligned}\end{equation} for any $t\in(0,t_{\phi}(\xi)]$. It implies that $T^t_{\phi}\xi\in {\rm Int}_{X} \tilde{C}_i$ for any $t\in(0,t_{\phi}(\xi)]$.  In fact, we obtain that
(i) for any $\phi\in K$ and $\xi\in \tilde{C}_i\setminus\{0\}$, if $(1-\tilde{\mu})\norm{Q_i v_{\phi}(s;\xi)}\leq\norm{P_i v_{\phi}(s;\xi)}$ holds on $[0,t_{\phi}(\xi)]$ for $\tilde{\mu}\leq 0$, then the estimates (\ref{cone-estimate-1})-(\ref{cone-estimate-2}) for the projection on $L_i$ along $L_i^c$ hold for $\tilde{\mu}\leq 0$ by replacing $\mu$ as $\tilde{\mu}$; (ii) for any $\phi\in K,\,\xi\in \tilde{C}_i\setminus\{0\}$ such that  (\ref{band-boundary}) holds on an interval of time $[0,\tilde{t}_{\phi}(\xi)]$, one has (\ref{cone-estimate-3}) holds. Therefore, we obtain that $T^t_{\phi}(\tilde{C}_{i}\setminus\{0\})\subset {\rm Int}_{X} \tilde{C}_i,\,t>0$ for any $\phi\in K$ and $i\in\mathbb{N}^+$ with $i\geq N$, and the proof have been completed.
\end{proof}

Let $I=\{0,1,2,\cdots,N\}$, $J=\{j_i:\,i\in I\}$ such that $j_i=N_i, i\geq 1$ and $j_0=\tilde{k}_{N}$. Let $C_i=\tilde{C}_{i+N}$ for any $i\in I$,  and denoted by $T_{\phi}=T^1_{\phi},\,\phi\in K$ for simplicity. Recall that $T(t)$ is a compact semigroup on $X$, and $\mathcal{F}\in C^2(X)$. It then follows from (\ref{const-varia-formu}) that $T_{\phi}$ is compact on $X$. Let $F=\Phi_1$ on $K\subset X^{\alpha}\subset X$ and $\mathcal{T}$ be the map from $K$ to $L(X)$, which is defined by $\mathcal{T}(\phi)=T_{\phi}$ for any $\phi\in K$. By virtue of \cite[Theorem 3.4.1]{Dan}, $\mathcal{T}$ is continous on $X$ and $F$ is continuous on $K$ in $X$. Here, we set an addtional assumption on $F$, that is $F$ is a homeomorphism on $K$. This assumption holds when the flow $\Phi_t$ adimits a flow extension (See e.g. \cite[Definition 2.6]{FWW-1} and \cite[Definition 2.6 and Part III]{Shen-Yi}) on its invaraint set. Then, we have the following theorem:

\begin{theorem}\label{DIB-Equa} The linear cocycle $(F,\mathcal{T})$ generated by (\ref{Equ:semi}), (\ref{Equ:Varia-Semi}) on $K\times X$ admits a decomposition of quasi-all invariant bundles with the index set $J$.
\end{theorem}
\begin{proof} This theorem is directly implied by Lemma \ref{cone-tilde}, \ref{Sp-cone-tilde} and Theorem A.
\end{proof}

\begin{rem}   A linear cocycle can be extracted from the linearized equation of an evolution equation along its solutions passing a point in its invariant set, or nonautonomous linear evolution equation with some almost-periodic functions being its coefficients. For the linear cocycle with a set of strongly invariant cones, the decomposition of quasi-all invaraint bundles plays a noticeable role in the analysis of local and global dynamics of related nonlinear systems or cocycles. 
\end{rem}

\end{document}